\documentclass[a4paper]{article}
\usepackage{amsfonts,amsmath,amssymb,amsthm,verbatim}

\newcommand{\f}[1]{\mathbf{#1}}
\newcommand{\tripleneg}{\negthickspace\negthickspace\negthickspace}
\newtheorem{thm}{Theorem}[section]
\newtheorem{lem}[thm]{Lemma}

\title{A Version of the Circle Method for the Representation of Integers by Quadratic Forms}

\author{Nic Niedermowwe}

\begin{document}

\maketitle

\section{Introduction}

It is a classical problem to determine the number of representations of a non-zero integer $N$
by a quadratic form $F(\f{x})$ with $\f{x}=(x_1,\ldots,x_n)$ contained in an expanding box $PB$,
where the real parameter $P$ tends to infinity.
The problem has its genesis in the degree two case of Waring's problem,
which Hardy and Littlewood \cite{HardyLittlewoodPNI} were able to treat with their newly invented circle
method in the case $n\ge 5$. A few years later Kloosterman \cite{Kloosterman} was able to improve on this by
giving a formula for the number of representations of an integer by definite diagonal quadratic
forms in at least four variables. Since then a variety of different proofs based on the circle
method has been given for the general problem. The present paper expounds yet another
variation of the argument.

In our approach we initially count the representations weighted by a Gaussian function.
This facilitates an essential application of Poisson summation, but causes additional
work if we are interested in the unweighted number of solutions.
In Theorems \ref{thmAsympFmla} and \ref{thmAsympFmlaChi} we obtain the asymptotic formula
\[ R_\omega(N)\sim I_\omega(N)\mathfrak{S}(N)\qquad\text{ as }\; P\longrightarrow\infty, \]
first for a Gaussian weight $\omega(\f{x})=w(\f{x})$ and then for the characteristic
function $w(\f{x})=\chi_{PB}(\f{x})$ on the region $PB$.
Here $I_\omega$ is the singular integral associated with $\omega$ and $\mathfrak{S}$ is the singular series.
It should be pointed out that Malyshev \cite{Malyshev} and Moroz \cite{Moroz} followed a similar strategy
to count integral points on quadrics. Also noteworthy is the work of Heath-Brown \cite{HB2} whose circle method
with compactly supported weights yields the precise order of magnitude for the number of representations.

Before getting started it will be necessary to introduce some notation and conventions.
Our results are valid whenever $n\ge 4$ and $F$ is non-singular.
We write $\mathcal{F}$ for the matrix of $F$ given by $F(\f{x})=\frac{1}{2}\f{x}^T\mathcal{F}\f{x}$,
and we let $\mathcal{M}$ be a real orthogonal matrix that diagonalizes $\mathcal{F}$.
Accordingly we choose an $n$-dimensional hyperrectangle $B$ such that
the edges of $\mathcal{M}^TB$ are parallel to the coordinate axes.
The letter $\varepsilon$ denotes an arbitrarily small positive number, not necessarily the same
from instance to instance. Implicit constants in big-$O$ and $\ll$ notation may depend
upon $F$, $B$ and $\varepsilon$.

\let\thefootnote\relax\footnotetext{2000 Mathematics Subject Classification: 11P55 (primary); 11D09, 11D45, 11D72, 11D85 (secondary).}

\vskip .2in
\emph{Acknowledgement.}
The author is studying for a D.Phil. at the University of Oxford. I would like to
thank Prof. Heath-Brown for the excellent supervision that has accompanied my work.
\vskip .2in

\section{Application of the circle method}\label{secCM}

In setting up the circle method we follow the paper of Heath-Brown \cite{HB1}. Thus we let
\begin{equation}\label{WeightDefn}
w(\f{x})=\pi^{-n/2}K^{An}\exp\left(-|\f{x}-P\f{x}_0|^2P^{-2}K^{2A}\right),
\end{equation}
where $K=\log P$, and $A$ is an arbitrarily large positive number that determines how good
the error term in the asymptotic formula for $R_{\chi_{PB}}(N)$ is going be.
The purpose of the parameter $\f{x}_0$ will become apparent in Section \ref{secTransition}.
For the time being all we need to know is that $\f{x}_0\ll 1$.
We set $\mathfrak{Q}(\f{x},N)=F(\f{x})-N$ and define
\begin{equation}\label{Salpha}
S(\alpha)=\sum_{\f{x}\in\mathbb{Z}^n}w(\f{x})e\big(\alpha\mathfrak{Q}(\f{x},N)\big).
\end{equation}
This sum is absolutely convergent, so that one may write
\begin{equation}\label{RwNDefn}
R_w(N)=\int_0^1S(\alpha)d\alpha=\mathop{\sum_{\f{x}\in\mathbb{Z}^n}}_{\mathfrak{Q}(\f{x},N)=0}w(\f{x}).
\end{equation}
We now follow the reasoning at the beginning of Section 3 of Heath-Brown's paper.
(There our $R_w(N)$ is called $I$.) By setting $Q=\lfloor P\rfloor$ instead of
$Q=\lfloor P^{3/2}\rfloor$ one obtains the equivalent of Heath-Brown's Lemma 7 for quadratic forms, namely

\begin{lem}\label{CFPlem7}
We have
\[ R_w(N)=\sum_{q\le Q}\int_{-1/qQ}^{1/qQ}S_0(q,z)dz\hskip 2.5in \]
\begin{equation}\label{lem7}
+O\Bigg(Q^{-2}\sum_{q\le Q}\sum_{|u|\le q/2}(1+|u|)^{-1}\max_{1/2\le qQ|z|\le 1}|S_u(q,z)|\Bigg),
\end{equation}
where
\[ S_u(q,z)=\mathop{\sum_{s=1}^q}_{(s,q)=1}e_q(us)S\left(\frac{\overline{s}}{q}+z\right). \]
\end{lem}

The next step is to separate the dependence of $S_u(q,z)$ on $u$ and $q$ from that on $z$.
We do so by using Poisson's summation formula, which states that a rapidly decreasing smooth
map $f$ is related to its Fourier transform by the identity
\[ \sum_{\f{x}\in\mathbb{Z}^n}f(\f{x})=\sum_{\f{x}\in\mathbb{Z}^n}\int_{\mathbb{R}^n}f(\f{y})e(\f{x}.\f{y})\,d\f{y}. \]
By the choice of our weight $w$ it is permissible to apply the above to the function
$f(\f{x})=w(\f{v}+q\f{x})e(\alpha\mathfrak{Q}(\f{v}+q\f{x},N))$. This yields
\[ S_u(q,z)=q^{-n}\sum_{\f{b}\in\mathbb{Z}^n}S_u(q,\f{b})I\left(z,\frac{\f{b}}{q}\right), \]
where
\begin{equation}\label{Suqb}
S_u(q,\f{b})=\mathop{\sum_{s=1}^q}_{(s,q)=1}\sum_{\f{v}\tripleneg\pmod{q}}
e_q\Big(\overline{s}\mathfrak{Q}(\f{v},N)+us-\f{b}.\f{v}\Big),
\end{equation}
\begin{equation}\label{Izbeta}
I(z,\boldsymbol{\beta})=\int_{\mathbb{R}^n}w(\f{x})e\Big(z\mathfrak{Q}(\f{x},N)
+\boldsymbol{\beta}\f{x}\Big)\,d\f{x}.
\end{equation}
An integration with respect to $d\f{x}$ is to be interpreted as an $n$-fold repeated integral $dx_1\ldots dx_n$.
The substitution $\f{x}\mapsto\mathcal{M}\f{x}$ transforms $F$ into a diagonal form $D$. Letting a superscript $^*$
indicate multiplication by $\mathcal{M}^T$ from the left, we have
\[ I(z,\boldsymbol{\beta})=e(-zN)\int_{\mathbb{R}^n}w(\mathcal{M}\f{x})
e\big(zD(\f{x})+\boldsymbol{\beta}^*\f{x}\big)\,d\f{x}. \]
Now $\mathcal{M}$ is orthogonal, so that it is possible to factorize the integral over $\mathbb{R}^n$ into $n$ terms,
each of the shape
\[ \pi^{-1/2}K^A\int_{\mathbb{R}}\exp\left(-(x-Px_0^*)^2P^{-2}K^{2A}+2\pi i(z\lambda x^2+\beta^*x)\right)dx, \]
where the $\lambda$ are the eigenvalues of $\frac{1}{2}\mathcal{F}$. The above expression introduces a convenient piece
of notation: vectors are represented by bold letters, and their components are denoted by the same letter in normal
print with an index. This index will often be omitted if it is irrelevant so long as it is kept fixed.
Returning to our proof, we write $-a$ for the coefficient of $x^2$, $b$ for the coefficient of $x$
and $c$ for the constant term in the argument of the exponential function.
Then the variable substitution $x\mapsto x+b/2a$ shows the above term to be equal to
\begin{equation}\label{Ifactor}
K^Aa^{-1/2}\exp\left(c+\frac{b^2}{4a}\right).
\end{equation}
We furthermore note that
\begin{equation}\label{Re}
\left|\exp\left(c+\frac{b^2}{4a}\right)\right|
=\exp\left(-\frac{\pi^2K^{2A}\left(2z\lambda x_0^*+\beta^*P^{-1}\right)^2}{P^{-4}K^{4A}+4\pi^2z^2\lambda^2}\right).
\end{equation}

We begin by proving two bounds on $I(z,\boldsymbol{\beta})$ that will be useful at various points
in the future.

\begin{lem}\label{lemIEst}
We have
\begin{equation}\label{IEst}
I(z,\boldsymbol{\beta})\ll K^{2An}\min\left(P^n,|z|^{-n/2}\right).
\end{equation}
\end{lem}

\begin{proof}
The lemma follows from the fact that $I(z,\boldsymbol{\beta})$ is the product of $n$ factors (\ref{Ifactor}),
each of which is $\ll K^{2A}\min(P,|z|^{-1/2})$. For it is easily seen from the definition of $a$ that
\[ a^{-1/2}\ll\min(P^{-2}K^{2A},|z|)^{-1/2} \]
and that $\exp(c+b^2/4a)\ll 1$ by (\ref{Re}).
\end{proof}

The next lemma equips us with a bound for $I(z,\boldsymbol{\beta})$ in terms of $\boldsymbol{\beta}$, which will 
subsequently be used to show that the sum $S_u(q,z)$ leaves an error of $O(1)$ when terminated at
\[ B_0=qK^s\left(P^{-1}+|z|P\right), \]
where $s=2+A$.

\begin{lem}
If $|\boldsymbol{\beta}|\ge B_0/q$ and $|z|\le 1$, then
\[ I(z,\boldsymbol{\beta})\ll\exp(-C\log^2PW), \]
where $W=2+|\boldsymbol{\beta}|$.
\end{lem}

\begin{proof}
We note that $|\boldsymbol{\beta}|\ge B_0/q$ implies $|\beta^*|\gg B_0/q$ for at least one $\beta^*$.
Fix such a $\beta^*$. Then
\[ \left|\frac{2z\lambda x_0^*}{P^{-1}\beta^*}\right|\le\frac{2|z\lambda x_0^*|}{CK^s\left(P^{-2}+|z|\right)}
\le\frac{2|\lambda x_0^*|}{CK^s}\le\frac{1}{2} \]
for all sufficiently large $P$, whence
\[ \left(2z\lambda x_0^*+P^{-1}\beta^*\right)^2\gg \left(P^{-1}\beta^*\right)^2. \]
Therefore $|c+b^2/4a|$ is
\begin{equation}\label{cba}
\gg\frac{K^{2A}P^{-2}\mathop{\beta^*}^2}{P^{-4}K^{4A}+z^2}.
\end{equation}
We now consider two cases. If $|\beta^*|\le P^2$, then $\log^2{PW}\ll K^4$, and the lemma follows
by noting that (\ref{cba}) is
\[ \gg\frac{K^{4(1+A)}(P^{-4}+z^2)}{P^{-4}K^{4A}+z^2}=\frac{P^{-4}+z^2}{P^{-4}+z^2K^{-4A}}K^4\ge K^4. \]
If on the other hand $|\beta^*|>P^2$, then it is sufficient to observe that (\ref{cba}) is
\[ \gg |\beta^*|^2P^{-2}\gg\log^2{|\beta^*|/P}\gg\log^2{\sqrt{|\beta^*|}}\gg\log^2{PW}. \]
\end{proof}

\begin{lem}\label{lemFinitebRange}
If $|z|\le 1$, then
\[ S_u(q,z)=q^{-n}\mathop{\sum_{\f{b}\in\mathbb{Z}^n}}_{|\f{b}|\le B_0}S_u(q,\f{b})I\left(z,\frac{\f{b}}{q}\right)+O(1). \]
\end{lem}

\begin{proof}
The trivial estimate $|S_u(q,z)|\le q^{n+1}$ together with the estimate of the previous lemma show that
\[ q^{-n}\mathop{\sum_{\f{b}\in\mathbb{Z}^n}}_{|\f{b}|\ge B_0}S_u(q,\f{b})I\left(z,\frac{\f{b}}{q}\right)
\ll q\sum_{\f{b}\in\mathbb{Z}^n}\exp\left(-C\log^2PW\right), \]
where now $W=2+|\f{b}|/q$, since the exponential function takes positive values only.
We split the summation into two. Using that the volume of the $n$-dimensional ball of radius $R$ is $\ll R^n$,
we get
\[ q\mathop{\sum_{\f{b}\in\mathbb{Z}^n}}_{|\f{b}|\le P}\exp\left(-C\log^2PW\right)\ll P^{n+1}P^{-C\log P} \ll 1, \]
since $q\le P$. The second part of the summation is
\[ q\mathop{\sum_{\f{b}\in\mathbb{Z}^n}}_{|\f{b}|\ge P}\exp\left(-C\log^2PW\right)
\ll q\sum_{k=0}^\infty\mathop{\sum_{\f{b}\in\mathbb{Z}^n}}_{2^kP\le |\f{b}|\le 2^{k+1}P}\exp\left(-C\log^2PW\right). \]
We observe that $\log PW\ge P(2+2^kP/q)\ge 2^kP$. Also, the number of integer points
$\f{b}$ with $2^kP\le |\f{b}|\le 2^{k+1}P$ is less than $2^{(k+2)n}P^n$, whence we can conclude our estimation with
\[ \ll \sum_{k=0}^\infty\left(2^kP\right)^{n-C\log 2^kP}\ll 1, \]
as claimed.
\end{proof}

Next we give estimates on the exponential sum $S_u(q,\f{b})$, which we shall denote by
$S_u(q,\f{b},N)$ if it is necessary to highlight its dependence on $N$.
First, one obtains the uniform bound
\begin{equation}\label{SimpleEst}
S_u(q,\f{b},N)\ll q^{n/2+1}
\end{equation}
by a simple adaptation of the proof of Lemma 25 of \cite{HB2}. Although this greatly improves on the
trivial estimate, in the case $n=4$ the bound is not strong enough for our purposes. However, we will
obtain a better bound if $q$ is square-free, which ultimately is sufficient as one can exploit
the fact that square-full numbers are relatively rare. More precisely, it is not difficult to see that
the number $\nu(X)$ of square-full numbers $x\le X$ satisfies
\begin{equation}\label{NoSqFree}
\nu(X)\ll X^{1/2}.
\end{equation}
Working towards a better estimate
we observe the multiplicative property 
\[ S_u(rs,\f{b},N)=S_u(r,\f{b},\overline{s}N)S_u(s,\f{b},\overline{r}N) \]
for coprime integers $r$ and $s$. From now on we shall write $q=q_1q_2$ where $q_1$ is the square-free
part of $q$ and $q_2$ is square-full.

\begin{lem}\label{lemAdvanEst}
We have
\begin{equation}\label{SuqbEst}
S_u(q,\f{b},N)\ll q_1^{(n+1)/2+\varepsilon}q_2^{n/2+1}(q_1,N)^{1/2}.
\end{equation}
\end{lem}

\begin{proof}
By multiplicativity we can factorize
\begin{eqnarray*}
S_u(q,\f{b},N)&=&S_u(q_1,\f{b},\overline{q_2}N)S_u(q_2,\f{b},\overline{q_1}N) \\
&=&S_u(q_2,\f{b},\overline{q_1}N)\prod_{i=1}^{\omega(q_1)}S_u(p_i,\f{b},\overline{q/p_i}N)
\end{eqnarray*}
If $p_i|2\det{\mathcal{F}}$, then $p_i\ll 1$ and
$S_u(p_i,\f{b},N)\ll 1$ trivially. Otherwise a straightforward adaptation of Lemma 26 of \cite{HB2}
shows that
\[ |S_u(p_i,\f{b},\overline{q/p_i}N)|
\le 2p_i^{(n+1)/2}\left(p_i\,,\,u-\overline{4}\f{b}^T\mathcal{F}^{-1}\f{b}\,,\,
\overline{q/p_i}N\right)^{1/2}. \]
On using that $2^{\omega(q_1)}=d(q_1)\ll {q_1}^\varepsilon$ we obtain
\[ S_u(q_1,\f{b},\overline{q_2}N)\ll q_1^{(n+1)/2+\varepsilon}(q_1,N)^{1/2}. \]
Combining this bound with the estimate
\[ S_u(q_2,\f{b},\overline{q_1}N)\ll q_2^{n/2+1} \]
from above proves the result.
\end{proof}

We have now gathered all the prerequisites needed to deduce (\ref{RwN}) from (\ref{lem7}).
First we deal with the error term.

\begin{lem}\label{lemET}
The big-$O$ term that occurs in (\ref{lem7}) is of order $O(P^{(n-1)/2+\varepsilon})$.
\end{lem}

\begin{proof}
By Lemma \ref{lemIEst} 
\[ I\left(z,\frac{\f{b}}{q}\right)\ll|z|^{-n/2}\ll(qQ)^{n/2}. \]
if $|z|\gg (qQ)^{-1}$. Consequently Lemma \ref{lemFinitebRange} shows that
\[ \max_{1/2\le qQ|z|\le 1}\left|S_u(q,z)\right|
\ll 1+q^{-n/2}Q^{n/2}\mathop{\sum_{\f{b}\in\mathbb{Z}^n}}_{|\f{b}|\ll K^s}\left|S_u(q,\f{b})\right|. \]
The first summand, $1$, produces a contribution of
\[ Q^{-2}\sum_{q\le Q}\sum_{|u|\le q/2}(1+|u|)^{-1}\ll 1 \]
to the error term of (\ref{lem7}).
The second summand's contribution to that error term is
\begin{equation}\label{SecSummand}
\mathop{\sum_{\f{b}\in\mathbb{Z}^n}}_{|\f{b}|\ll K^s}Q^{n/2-2}\sum_{q\le Q}q^{-n/2}\sum_{|u|\le q/2}(1+|u|)^{-1}|S_u(q,\f{b})|.
\end{equation}
We replace the summation over the $\f{b}$ by a factor $O(P^\varepsilon)$.
This is justified because any estimates on $S_u(q,\f{b})$ we shall use will be independent of $\f{b}$.
Also applying Lemma \ref{lemAdvanEst} then yields
\[ \ll P^\varepsilon Q^{n/2-2}\sum_{q_2\le Q}q_2\sum_{q_1\le Q/q_2}{q_1}^{1/2}(q_1,N)^{1/2}. \]
At this stage we note that
\begin{equation}\label{DivisorEstimate}
\sum_{r\le X}(r,N)=\sum_{d|N}\mathop{\sum_{r\le X}}_{(r,N)=d}d
\le d(N)X\ll P^\varepsilon X,
\end{equation}
whereupon partial summation yields
\[ \sum_{r\le X}r^{1/2}(r,N)^{1/2}\ll P^\varepsilon X^{3/2}. \]
Therefore we can continue our estimation as
\[ \ll P^\varepsilon Q^{n/2-2}\sum_{q_2\le Q}q_2\left(\frac{Q}{q_2}\right)^{3/2}
=P^\varepsilon Q^{(n-1)/2}\sum_{q_2\le Q}q_2^{-1/2}\ll P^{(n-1)/2+\varepsilon}, \]
where the last step follows from partial summation in connection with (\ref{NoSqFree}).
\end{proof}

We now turn our attention to the main term
of formula (\ref{lem7}), which by Lemma \ref{lemFinitebRange} equals
\begin{equation}\label{MainTerm}
\sum_{q\le Q}\int_{-1/qQ}^{1/qQ}q^{-n}\mathop{\sum_{\f{b}\in\mathbb{Z}^n}}_{|\f{b}|\le B_0}
S_0(q,\f{b})I\left(z,\frac{\f{b}}{q}\right)\,dz+O(1)
\end{equation}
All terms corresponding to non-zero $\f{b}$ can be absorbed into the error term.

\begin{lem}\label{lemMTC}
The main term (\ref{MainTerm}) of formula (\ref{lem7}) is equal to
\begin{equation}\label{MTC}
\sum_{q\le Q}q^{-n}S_0(q,\f{0})\int_{-1/qQ}^{1/qQ}I\left(z,\f{0}\right)dz
+O\left(P^{(n-1)/2+\varepsilon}\right).
\end{equation}
\end{lem}

\begin{proof}
If $z$ and $q$ are simultaneously confined to the ranges
\begin{equation}\label{Ranges}
|z|<(2q(Q+1)K^s)^{-1}\quad\text{and}\quad q<Q/2K^s,
\end{equation}
then $B_0<1$, forcing $\f{b}=\f{0}$. It remains to show that the contribution from the ranges where (\ref{Ranges})
does not hold is of order $O(P^{(n-1)/2+\varepsilon})$.

In the case where
\[ z\in I=\left[\frac{-1}{qQ},\frac{1}{qQ}\right]\setminus\left(\frac{-K^{-s}}{2q(Q+1)},\frac{K^{-s}}{2q(Q+1)}\right), \]
the relevant part of the main term can be estimated by Lemma \ref{lemIEst} as
\[ \ll\sum_{q\le Q}\mathop{\sum_{\f{b}\in\mathbb{Z}^n}}_{|\f{b}|\ll K^s}
q^{-n}\left|S_0(q,\f{b})\right|\int_IK^{2An}|z|^{-n/2}\,dz \]
Yet again making sure that only estimates independent of $\f{b}$ will be used for $S_0(q,\f{b})$,
the summation over the $\f{b}$ simply contributes a factor $O(P^\varepsilon)$, i.e.
\[ \ll P^\varepsilon\sum_{q\le Q}q^{-n}\left|S_0(q,\f{b})\right|(qQK^s)^{n/2-1}
\ll P^{n/2-1+\varepsilon}\sum_{q\le Q}q^{-n/2-1}\left|S_0(q,\f{b})\right|. \]
We must show that
\begin{equation}\label{qSumEst}
\sum_{q\le Q}q^{-n/2-1}\left|S_0(q,\f{b})\right|\ll Q^{1/2+\varepsilon}.
\end{equation}
Indeed, by Lemma \ref{lemAdvanEst} the sum is
\[ \ll\sum_{q\le Q}q_1^{-1/2+\varepsilon}(q_1,N)^{1/2}
\ll\sum_{q_2\le Q}\sum_{q_1\le Q/q_2}q_1^{-1/2+\epsilon}(q_1,N)^{1/2}, \]
which by (\ref{DivisorEstimate}) and partial summation can be seen to be
\[ \ll P^{1/2+\epsilon}\sum_{q_2\le Q}q_2^{-1/2}\ll P^{1/2+\epsilon}. \]

The other case we need to consider is $q\ge Q/2K^s$. Here we estimate the corresponding part
of the main term by Lemmata \ref{lemIEst} and \ref{lemAdvanEst} as
\begin{eqnarray*}
&\ll& P^\varepsilon\sum_{\frac{Q}{2K^s}\le q\le Q}q^{-n}|S_0(q,\f{b})|
\int_{\frac{-1}{qQ}}^{\frac{1}{qQ}}K^{An}P^ndz. \\
&\ll& P^{n-1+\varepsilon}\sum_{q_2\le Q}q_2^{-n/2}\sum_{\frac{Q}{2q_2K^s}\le q_1\le \frac{Q}{q_2}}q_1^{-(n+1)/2+\varepsilon}
(q_1,N)^{1/2}.
\end{eqnarray*}
As before we use (\ref{DivisorEstimate}) in a partial summation to estimate the sum over $q_1$ as
\[ \ll P^\varepsilon\frac{Q}{q_2}\left(\frac{Q}{q_2K^s}\right)^{-(n+1)/2}, \]
which leaves us with a bound of 
\[ \ll P^{(n-1)/2+\varepsilon}\sum_{q_2\le Q}q_2^{-1/2}\ll P^{(n-1)/2+\varepsilon} \]
for the whole expression.
\end{proof}

By extending the summation to infinity and the interval of integration to the whole
of $\mathbb{R}$ in (\ref{MTC}), we arrive at the singular series and singular integral.

\begin{thm}\label{thmAsympFmla}
We have
\begin{equation}\label{RwN}
R_w(N)=I_w(N)\mathfrak{S}(N)+O\left(P^{(n-1)/2+\varepsilon}\right)\quad\text{as }P\longrightarrow\infty.
\end{equation}
Here $I_w$ is the singular integral
\begin{equation}\label{SingIntDefn}
I_w(N)=\int_{-\infty}^\infty I(z,\f{0})\,dz,
\end{equation}
and $\mathfrak{S}$ is the singular series 
\begin{equation}\label{SingSerDefn}
\mathfrak{S}(N)=\sum_{q=1}^\infty q^{-n}S_0(q,\f{0}).
\end{equation}
\end{thm}

\begin{proof}
Extending the integration in (\ref{MTC}) to the whole of the real line produces by Lemma \ref{lemIEst} an error of
\[ \ll P^\varepsilon\sum_{q\le Q}\frac{|S_0(q,\f{0})|}{q^n}\int_{\mathbb{R}\setminus[\frac{-1}{qQ},\frac{1}{qQ}]}|z|^{-n/2}\,dz
\ll P^\varepsilon Q^{n/2-1}\sum_{q\le Q}q^{-n/2-1}|S_0(q,\f{0})|. \]
This is indeed of the required size $O(P^{(n-1)/2+\varepsilon})$ as can be seen from (\ref{qSumEst}).

Next we extend the summation to infinity, which gives an error of
\[ |I_w(N)|\sum_{q>Q}q^{-n}|S_0(q,\f{0})|. \]
The first factor satisfies
\begin{equation}\label{IwNEst}
I_w(N)\ll P^{n-2+\varepsilon},
\end{equation}
since Lemma \ref{lemIEst} shows that
\[ \int_{-P^{-2}}^{P^{-2}}I(z,\f{0})\,dz\ll\int_{-P^{-2}}^{P^{-2}}K^{2An}P^n\,dz\ll P^{n-2+\varepsilon} \]
as well as
\[ \int_{\mathbb{R}\setminus [-P^{-2},P^{-2}]}I(z,\f{0})\,dz
\ll K^{2An}\int_{\mathbb{R}\setminus [-P^{-2},P^{-2}]}|z|^{-n/2}\,dz\ll P^{n-2+\varepsilon}. \]
The theorem will follow if we can show that
\begin{equation}\label{SingSerUpper}
\sum_{q>Q}q^{-n}|S_0(q,\f{0})|\ll Q^{(3-n)/2+\epsilon}.
\end{equation}
Indeed, splitting the sum into dyadic ranges and applying Lemma \ref{lemAdvanEst} yields the bound
\begin{equation}\label{term}
\ll\sum_{l=0}^\infty\left(2^lQ\right)^{-n}\sum_{2^lQ<q\le 2^{l+1}Q}q_1^{(n+1)/2+\varepsilon}q_2^{n/2+1}(q_1,N)^{1/2}.
\end{equation}
Now the inner sum is
\[ \ll (2^{l+1}Q)^{(n+1)/2+\epsilon}\sum_{q_2\le 2^{l+1}Q}q_2^{1/2}\sum_{q_1\le 2^{l+1}Q/q_2}{q_1}^\epsilon (q_1,N)^{1/2}. \]
Moreover, estimate (\ref{DivisorEstimate}) together with partial summation shows that the sum over $q_1$ is
\[ \ll (2^{l+1}Q)^{1+\epsilon}{q_2}^{-1} \]
by our choice of $Q$. Therefore we obtain the required bound
\[ \ll\sum_{l=0}^\infty\left(2^lQ\right)^{-n}(2^{l+1}Q)^{(n+3)/2+\epsilon}\ll Q^{(3-n)/2+\epsilon} \]
for the summation over all $q>Q$.
\end{proof}

\section{Transition to characteristic weight}\label{secTransition}

We begin by introducing some new notation.
Given any function $\omega:\mathbb{R}^n\rightarrow\mathbb{R}$ we write $R_\omega$ and $I_\omega$ for
the (formal) functions obtained by replacing $w(\f{x})$ with $\omega(\f{x})$ in the definitions of $R_w$
and $I_w$ respectively. We write $\f{c}$ for the centre of $B$, and set $\Gamma=B-\f{c}$. 
By integrating $w(\f{x})$ with respect to $\f{x}_0$ we obtain the two weight functions 
\begin{equation}\label{WpmDefn}
W_\pm(\f{x})=\int_{(1\pm K^{-A/2})\Gamma+\f{c}}w(\f{x},\f{x}_0)\,d\f{x}_0.
\end{equation}
We shall see below that, up to a small error, the weight $W_+$ majorises $\chi_{PB}$,
whereas $W_-$ essentially minorises it. It follows from equation (\ref{RwNDefn}) that
\begin{equation}\label{RWNDefn}
R_{W_\pm}(N)=\int_{(1\pm K^{-A/2})\Gamma+\f{c}}R_w(N)\,d\f{x}_0
\end{equation}
and that this is indeed the number of integral zeros $\f{x}$ of $F(\f{x},N)$ weighted by $W_\pm$.
Let us record the following corollary of Theorem \ref{thmAsympFmla}.

\begin{lem}\label{lemAsympWpm}
We have
\[ R_{W_\pm}(N)=I_{W_\pm}(N)\mathfrak{S}(N)+O\left(P^{(n-1)/2+\varepsilon}\right). \]
\end{lem}

\begin{proof}
The result follows by integrating (\ref{RwN}) with respect to $\f{x}_0$ over the box $(1\pm K^{-A/2})\Gamma+\f{c}$.
By (\ref{RWNDefn}) it suffices to note that
\[ I_{W_\pm}(N)=\int_{(1\pm K^{-A/2})\Gamma+\f{c}}I_w(N)\,d\f{x}_0, \]
which holds because we may swap the integrations over $\f{x}_0$ and $z$  by (\ref{IwNEst}), and then the ones over
$\f{x}_0$ and $\f{x}$ by Lemma \ref{lemIEst}.
\end{proof}

\begin{lem}\label{lemWToChi}
For any $C>0$,
\[ R_{W_+}(N)\ge R_{\chi_{PB}}(N)+O(P^{-C})\ge R_{W_-}(N). \]
\end{lem}

\begin{proof}
Recall that the box $\Gamma^*$ is centred at the origin and has edges parallel to the coordinate axes.
Therefore we can write $\Gamma^*=\prod_{i=1}^n[-\gamma_i^*,\gamma_i^*]$ with $\gamma>0$.
Substituting $\f{x}_0\mapsto\mathcal{M}P^{-1}(PK^{-A}\f{x}_0+\f{x}^*)$ in (\ref{WpmDefn}) gives
\begin{equation}\label{Wpm}
W_\pm(\f{x}^*)
=\pi^{-n/2}\prod_1^n\int_{K^A(-(1\pm K^{-A/2})\gamma^*+c^*-x^*/P)}^{K^A((1\pm K^{-A/2})\gamma^*+c^*-x^*/P)}
e^{-{x_0}^2}\,dx_0.
\end{equation}
We deduce that $W_\pm(\f{x}^*)$ is positive and less than $1$, and that it has a unique global
maximum at $\f{x}^*=P\f{c}^*$. Furthermore, it is easily seen that $W_\pm(\f{x}^*)$
decreases whenever all but one of the coordinates of $\f{x}^*$ are fixed and the distance
of the remaining coordinate $x_i$, say, from $Pc_i^*$ increases.

Hence on the hyperrectangle $PB^*$ the function $W_+(\f{x}^*)$ attains its minima
in the corners of $PB^*$. By symmetry, these minima are of same size, so that in order to derive a lower
bound for $W_+(\f{x}^*)$ on $PB^*$ it suffices to consider the corner $P(\f{c}^*-\boldsymbol{\gamma}^*)$.
By virtue of the inequality
\begin{equation}\label{GaussInt}
\frac{\sqrt{\pi}}{2}(1-e^{-\alpha^2})<\int_0^\alpha e^{-x^2}\,dx,\quad\forall\alpha>0
\end{equation}
applied with $\alpha=K^{A/2}\gamma^*$ we obtain
\[ W_+(P(\f{c}^*-\boldsymbol{\gamma}^*))>\prod_1^n\left(1-\exp\left(-K^A{\gamma^*}^2\right)\right)\ge 1+O(P^{-C}). \]
The first part of the lemma's assertion follows by (\ref{RWNDefn}) and the observation that there are
only $O(P^n)$ integer points in $PB$.

We deal with the second inequality by noting that on $\mathbb{R}^n\setminus\text{int}(PB)$
the function $W_-(\f{x}^*)$ attains its maximum at one of the centres
\[ \f{c}_j^\pm=P(\f{c}^*\pm(0,\ldots,0,\gamma_j^*,0,\ldots,0)) \]
of the faces of the hyperrectangle $PB^*$. We would like to find an upper bound for $W_-(\f{x}^*)$
on $\mathbb{R}^n\setminus PB$, so by symmetry it is sufficient to consider the $\f{c}_j^-$ only.
We let $\f{x}^*=\f{c}_j^-$ in (\ref{Wpm}). On estimating each factor other than the $j$th as $\ll 1$, we obtain
\[ W_-(\f{x}^*)\ll\int_{K^{A/2}\gamma_j^*}^\infty e^{-x_0^2}\,dx_0
\ll\frac{\sqrt{\pi}}{2}-\int_0^{K^{A/2}\gamma_j^*}e^{-x_0^2}\,dx_0 \]
On using (\ref{GaussInt}) with $\alpha=K^{A/2}\gamma_j^*$ we get
\[ W_-(\f{x}^*)\ll\exp\left(-K^A{\gamma_j^*}^2\right)\ll P^{-C}. \]
Hence we can deduce that
\[ \mathop{\sum_{\f{x}\in\mathbb{Z}^n}}_{\f{x}\not\in PB}W_-(\f{x})
=O(P^{-C})+\mathop{\sum_{\f{x}\in\mathbb{Z}^n}}_{|\f{x}-P\f{c}|\ge CP^2}W_-(\f{x}), \]
as there are $O(P^{2n})$ integer points in an $n$-dimensional ball of radius $CP^2$.
From the definition of $W_-(\f{x})$ it follows that $|\f{x}-P\f{c}|\ge C2^kP^2$ implies
\[ W_-(\f{x})\ll\exp(-C2^{2k-2}P^2K^{2A}) \]
since $\f{x}_0\ll 1$. Therefore
\[ \mathop{\sum_{\f{x}\in\mathbb{Z}^n}}_{|\f{x}-P\f{c}|\ge CP^2}W_-(\f{x})
\ll P^{-C}+\sum_{k=1}^\infty\mathop{\sum_{\f{x}\in\mathbb{Z}^n}}_{C2^kP^2\le |\f{x}-P\f{c}|<C2^{k+1}P^2}
\exp(-CkK)\ll P^{-C}. \]
Thus
\[ R_{W_-}(N)=\mathop{\sum_{\f{x}\in\mathbb{Z}^n}}_{\f{x}\in PB}W_-(\f{x})
+\mathop{\sum_{\f{x}\in\mathbb{Z}^n}}_{\f{x}\not\in PB}W_-(\f{x})\le R_{\chi_{PB}}(N)+O(P^{-C}), \]
as required.
\end{proof}

The following lemma provides us with a first derivative estimate on certain exponential integrals.

\begin{lem}\label{lemDerivEst}
For $z\in\mathbb{R}$ we have
\begin{equation}\label{DervEst}
\int_a^be(z v^2)\,dv\ll |z|^{-1/2}
\end{equation}
uniformly in all $a$, $b\in\mathbb{R}$.
\end{lem}

\begin{proof}
We estimate the integral over the range $[-|z|^{-1/2},|z|^{1/2}]\cap [a,b]$ trivially, i.e. as $\ll |z|^{-1/2}$.
Hence we may assume without loss of generality that $|z|^{-1/2}\le a<b$. We write
\[ \int_a^be(z v^2)\,dv=\int_a^b4\pi izv\,e(zv^2)\frac{1}{4\pi izv}\,dv, \]
which by partial integration is seen to be
\[ \ll\frac{1}{|zb|}+\frac{1}{|za|}+\frac{1}{|z|}\int_a^b\frac{1}{v^2}\,dv\ll |z|^{-1/2}, \]
as was required.
\end{proof}

We also note that (\ref{SimpleEst}) and (\ref{SingSerUpper}) give us the simple upper bound
\begin{equation}\label{SingSerBound}
\mathfrak{S}(N)\ll \log N + P^{-1/2+\varepsilon}\ll K.
\end{equation}
To avoid cluttering our notation we let $R=[-K^{A/3},K^{A/3}]$ in what follows.

\begin{lem}\label{lemChiToChi}
We have
\[ I_{\chi_{P((1\pm K^{-A/2})\Gamma+\f{c})}}(N)\mathfrak{S}(N)=I_{\chi_{PB}}(N)\mathfrak{S}(N)+O(P^{n-2}K^{-A/6+1}). \]
\end{lem}

\begin{proof}
By (\ref{SingSerBound}) it is plainly sufficient to show that
\[ I_{\chi_{P((1\pm K^{-A/2})\Gamma+\f{c})}}(N)-I_{\chi_{PB}}(N)\ll K^{-A/6}. \]
Defining
\[ V_+=(1+K^{-A/2})\Gamma^*+\f{c}^*\setminus B^*\quad\text{ and }\quad
V_-=B^*\setminus (1-K^{-A/2})\Gamma^*+\f{c}^* \]
we have
\begin{equation}\label{IDiff}
I_{\chi_{P((1\pm K^{-A/2})\Gamma+\f{c})}}(N)-I_{\chi_{PB}}(N)
\ll P^{n-2}\int_{-\infty}^\infty\left|\int_{V_\pm}e(zF(\f{x}))\,d\f{x}\right|\,dz.
\end{equation}
The inner integral is trivially $\ll K^{-A/2}$, so that the
integration over $z$ restricted to $R$ is $\ll K^{-A/6}$.
Furthermore, the modulus in (\ref{IDiff}) is
\[ \le\left|\int_{(1\pm K^{-A/2})\Gamma^*+\f{c}^*}e(zF(\f{x}))\,d\f{x}\right|+\left|\int_{B^*}e(zF(\f{x}))\,d\f{x}\right|. \]
We recall that the substitution $\f{x}\mapsto\mathcal{M}\f{x}$ diagonalizes $F$ and aligns the edges of $B$ with
the coordinate axes, allowing us to factorize the integrals. Thereupon each factor can be estimated according
to Lemma \ref{lemDerivEst}. This yields
\[ \int_{\mathbb{R}\setminus R}\left|\int_{V_\pm}e(zD(\f{x}))\,d\f{x}\right|\,dz
\ll\int_{\mathbb{R}\setminus R}|z|^{-n/2}\, dz\ll K^{A(2-n)/6}\ll K^{-A/6}, \]
as required.
\end{proof}

\begin{lem}\label{lemChiToW}
We have
\begin{equation}\label{midchain}
I_{\chi_{P((1\pm K^{-A/2})\Gamma+\f{c})}}(N)\mathfrak{S}(N)=R_{W_\pm}(N)+O\left(P^{n-2}K^{-A/3+1}\right).
\end{equation}
\end{lem}

\begin{proof}
By Lemma \ref{lemAsympWpm} and (\ref{SingSerBound}) it is sufficient to prove
\begin{equation}\label{IDiff2}
I_{\chi_{P((1\pm K^{-A/2})\Gamma+\f{c})}}(N)-I_{W_\pm}(N)\ll P^{n-2}K^{-A/6}.
\end{equation}
If we let $\delta=K^{-A}$, then by definition
\[I_{W_\pm}(N)=\frac{\delta^{-n}}{\pi^{n/2}}\int_{-\infty}^\infty\int_{(1\pm\delta^{1/2})\Gamma+\f{c}}\int_{\mathbb{R}^n}
w(\f{x},\f{x}_0)e\big(z\mathfrak{Q}(\f{x},N)\big)\,d\f{x}\,d\f{x}_0\,dz. \]
So substituting $\f{x}\mapsto P(\delta\f{x}+\f{x}_0)$ and $z\mapsto P^{-2}z$ gives
\[ I_{W_\pm}(N)=\frac{P^{n-2}}{\pi^{n/2}}\int_{\mathbb{R}}\int_{(1\pm K^{-A/2})\Gamma+\f{c}}J(\delta,N/P^2)\,d\f{x}_0\,dz, \]
where
\[ J(\alpha,\beta)=\int_{\mathbb{R}^n}e^{-|\f{x}|^2}e(z\mathfrak{Q}(\alpha\f{x}+\f{x}_0,\beta))\,d\f{x}. \]
On the other hand we have
\[ I_{\chi_{P((1\pm K^{-A/2})\Gamma+\f{c})}}(N)
=\frac{P^{n-2}}{\pi^{n/2}}\int_{\mathbb{R}}\int_{(1\pm K^{-A/2})\Gamma+\f{c}}
J(0,N/P^2)\,d\f{x}_0\,dz, \]
whence (\ref{IDiff2}) becomes
\begin{equation}\label{JDiff}
\int_{-\infty}^\infty\int_{(1\pm K^{-A/2})\Gamma+\f{c}}J(\delta,N/P^2)-J(0,N/P^2)\,d\f{x}_0\,dz\ll K^{-A/6}.
\end{equation}
The difference $J(\delta,N/P^2)-J(0,N/P^2)$ may be written as
\[ \int_{\mathbb{R}^n}\frac{e\big(z\mathfrak{Q}(\f{x}_0,N/P^2)\big)}{e^{|\f{x}|^2}}
\Big(e\big(zF(\delta\f{x}+\f{x}_0)-zF(\f{x}_0)\big)-1\Big)\,d\f{x} \]
and is hence bounded by
\[ \ll\int_{\mathbb{R}^n}e^{-|\f{x}|^2}\big|z\big(F(\delta\f{x}+\f{x}_0)-F(\f{x}_0)\big)\big|\,d\f{x} \]
since $e(y)-1\ll |y|$. Moreover,
\[ F(\delta\f{x}+\f{x}_0)-F(\f{x}_0)=F(\delta\f{x})+\delta\f{x}.\nabla F(\f{x}_0)\ll \delta^2|\f{x}|^2+\delta|\f{x}|, \]
which shows that
\[ J(\delta,N/P^2)-J(0,N/P^2)\ll |z|\delta. \]
Therefore
\begin{equation}\label{JLow}
\int_R\int_{(1\pm K^{-A/2})\Gamma+\f{c}}J(\delta,N/P^2)-J(0,N/P^2)\,d\f{x}_0\,dz\ll\ K^{-A/3}.
\end{equation}
To estimate the remaining range $\mathbb{R}\setminus R$ of the integration over $z$ in (\ref{JDiff})
we note that uniformly in all $\alpha\in\mathbb{R}$
\begin{equation}\label{SingIntBound}
\int_{(1\pm\sqrt{\alpha})\Gamma+\f{c}}e\big(z\mathfrak{Q}(\alpha\f{x}+\f{x}_0,N/P^2)\big)\,d\f{x}_0\ll |z|^{-n/2},
\end{equation}
which can be seen by pulling out the term $e(-zN)$, diagonalizing $F$, factorizing the integrand
and then applying Lemma \ref{lemDerivEst}. Hence
\[ \int_{\mathbb{R}\setminus R}\int_{(1+\sqrt{\alpha})\Gamma+\f{c}}J(\alpha,N/P^2)\,d\f{x}_0\,dz
\ll\int_{\mathbb{R}\setminus R}|z|^{-n/2}\,dz\ll K^{-A/3}. \]
Applying this result with $\alpha=0$ and $\alpha=\delta$ shows together with (\ref{JLow}) that
the estimate (\ref{JDiff}) holds.
\end{proof}

\begin{thm}\label{thmAsympFmlaChi}
We have
\[ R_{\chi_{PB}}(N)=I_{\chi_{PB}}(N)\mathfrak{S}(N)+O\left(P^{n-2}K^{-A/6+1}\right). \]
\end{thm}

\begin{proof}
Combining Lemmata \ref{lemWToChi}, \ref{lemChiToChi} and \ref{lemChiToW} yields the chain of (in)equalities
\[ I_{\chi_{PB}}\mathfrak{S}=I_{\chi_{P((1+K^{-A/2})\Gamma+\f{c})}}\mathfrak{S}+O(P^{n-2}K^{-A/6+1})
=R_{W_+}\ge R_{\chi_{PB}}+O(P^{-C})\]
\[ \ge R_{W_-}=I_{\chi_{P((1-K^{-A/2})\Gamma+\f{c})}}\mathfrak{S}+O(P^{n-2}K^{-A/6+1})=I_{\chi_{PB}}\mathfrak{S}. \]
This proves the theorem.
\end{proof}

\noindent
Nic Niedermowwe \\
Mathematical Institute \\
24--29 St. Giles' \\
Oxford \\
OX1 3LB \\
United Kingdom \\
niedermo@maths.ox.ac.uk \\
\\
\\

\end{document}